\pdfoutput=1
\RequirePackage{ifpdf}
\ifpdf 
\documentclass[pdftex]{sigma}
\else
\documentclass{sigma}
\fi

\usepackage{amscd}
\usepackage{ytableau}
\usepackage{soul}
\usepackage{xargs}

\numberwithin{equation}{section}

\newtheorem{Theorem}{Theorem}[section]
\newtheorem*{Theorem*}{Theorem}
\newtheorem{Corollary}[Theorem]{Corollary}

\newtheorem{Proposition}[Theorem]{Proposition}
\newtheorem{Claim}[Theorem]{Claim}
 { \theoremstyle{definition}

\newtheorem{Remark}[Theorem]{Remark}
\newtheorem{Problem}[Theorem]{Problem}
}

\def\({\left(}
\def\){\right)}

\def\lp{\left(}
\def\rp{\right)}

\begin{document}

\newcommand{\arXivNumber}{2408.16902}

\renewcommand{\thefootnote}{}

\renewcommand{\PaperNumber}{026}

\FirstPageHeading

\ShortArticleName{Zeros of Hook Polynomials and Related Questions}

\ArticleName{Zeros of Hook Polynomials and Related Questions\footnote{This paper is a~contribution to the Special Issue on Basic Hypergeometric Series Associated with Root Systems and Applications in honor of Stephen C.~Milne's 75th birthday. The~full collection is available at \href{https://www.emis.de/journals/SIGMA/Milne.html}{https://www.emis.de/journals/SIGMA/Milne.html}}}

\Author{Walter BRIDGES~$^{\rm a}$, William CRAIG~$^{\rm b}$, Amanda FOLSOM~$^{\rm c}$ and Larry ROLEN~$^{\rm d}$}

\AuthorNameForHeading{W.~Bridges, W.~Craig, A.~Folsom and L.~Rolen}

\Address{$^{\rm a)}$~University of North Texas, USA}
\EmailD{\href{mailto:Walter.Bridges@unt.edu}{Walter.Bridges@unt.edu}}

\Address{$^{\rm b)}$~United States Naval Academy, USA}
\EmailD{\href{mailto:wcraig@usna.edu}{wcraig@usna.edu}}

\Address{$^{\rm c)}$~Amherst College, USA}
\EmailD{\href{mailto:afolsom@amherst.edu}{afolsom@amherst.edu}}

\Address{$^{\rm d)}$~Vanderbilt University, USA}
\EmailD{\href{mailto:larry.rolen@vanderbilt.edu}{larry.rolen@vanderbilt.edu}}

\ArticleDates{Received September 03, 2024, in final form April 10, 2025; Published online April 21, 2025}

\Abstract{We study the zero set of polynomials built from partition statistics, complementing earlier work in this direction by Boyer, Goh, Parry, and others. In particular, addressing a question of Males with two of the authors, we prove asymptotics for the values of $t$-hook polynomials away from an annulus and isolated zeros of a theta function. We also discuss some open problems and present data on other polynomial families, including those associated to deformations of Rogers--Ramanujan functions.}


\Classification{05A17; 30C15; 05A15; 11P82; 11F27}

\renewcommand{\thefootnote}{\arabic{footnote}}
\setcounter{footnote}{0}

\begin{flushright}
\begin{minipage}{58mm}
\it Dedicated to Stephen C.~Milne\\ on the occasion of his 75th birthday
\end{minipage}
\end{flushright}

\section{Introduction}

For a positive integer $n$, a {\it partition} of $n$ is a non-increasing sequence $\lambda = \lp \lambda_1, \dots, \lambda_\ell \rp$ of positive integers such that $|\lambda| := \lambda_1 + \dots + \lambda_\ell = n$. We write $\lambda \vdash n$ to denote that $\lambda$ is a partition of $n$, and we let $\ell = \ell\lp \lambda \rp$ be the number of parts of $\lambda$. Since the very beginning of partition theory, generating functions have been an essential tool. In the last century, the {\it circle method}
as developed first by Hardy and Ramanujan \cite{HR} has become an indispensable tool in studying the coefficients of these generating functions as well as many counting problems in additive combinatorics and number theory, including the Waring problem and weak Goldbach conjecture. Hardy and Ramanujan used the generating function for the {\it partition function} $p(n) := \#\{ \lambda \vdash n \}$, given for $|q|<1$ by
\begin{align*}
 \sum_{n \geq 0} p(n) q^n = \prod_{n \geq 1} \dfrac{1}{1 - q^n},
\end{align*}
in order to prove an asymptotic series for $p(n)$ as $n \to \infty$, whose main term yields
\begin{align*}
 p(n) \sim \dfrac{1}{4n\sqrt{3}} {\rm e}^{\pi\sqrt{\frac{2n}{3}}}.
\end{align*}

We will consider applications of the circle method to partition statistics. An important example in this respect is the number of parts $\ell\lp \lambda \rp$ of partitions. One can keep track of both the size and number of parts simultaneously in generating function form by using multiple variables; indeed, it is not hard to see that
\begin{align*}
 \sum_{\lambda} w^{\ell\lp\lambda\rp} q^{|\lambda|} = \prod_{n \geq 1} \dfrac{1}{1 - w q^n}.
\end{align*}
Properties of the parts of partitions are then encoded by the polynomials
\begin{equation}\label{E:QnwDEF}Q_n(w) := \sum_{\lambda \vdash n} w^{\ell\lp\lambda\rp},\end{equation}
which also form the power series coefficients in the $q$-variable of the above generating function. Indeed, every statistic on partitions possesses its own variation on the $Q_n(w)$ polynomials, and these polynomials can encode deep information about partitions. For example, the famous Ramanujan congruences
\begin{align*}
 p(5n+4) \equiv 0 \pmod{5}, \qquad p(7n+5) \equiv 0 \pmod{7}, \qquad p(11n+6) \equiv 0 \pmod{11}
\end{align*}
can be explained directly in terms of the divisibility of the {\it rank} and {\it crank polynomials} by cyclotomic polynomials for $n$ in the given arithmetic progressions (see, for instance, \cite{AG,BGRT,Garvan}).

This example illustrates the core philosophy of this paper: to study partition statistics by investigating the corresponding polynomials. In addition to divisibility relations, perhaps the most natural property of a set of polynomials to study is its set of roots. In recent years, there has been growing interest in studying roots of polynomials that encode data from number-theoretic objects. For instance, zeros of period polynomials of newforms have been shown by Jin, Ma, Ono, and Soundararajan to satisfy a ``Riemann hypothesis'' \cite{JMOS}. Another example is the work of Griffin, Ono, Zagier, and one of the authors \cite{GORZ} showing that the so-called Jensen polynomials encoded by integer partition numbers eventually have real roots. This directly implies that an infinite sequence of inequalities are eventually satisfied by partition numbers.
 From this point of view, it is natural to study the zeros of sequences of partition polynomials formed from important partition statistics. Indeed, this project was partly inspired by Boyer and Goh's fascinating study of the {\it zero attractor}\footnote{Roughly speaking, this is the set of limit points in $\mathbb{C}$ of the zeros of a sequence of polynomials.} of the number-of-parts polynomials~$Q_n(w)$~\mbox{\cite{BG1,BG2}} (see also \cite{BoyerParry1, BoyerParry3,BoyerParry2,BoyerParry4,BFG,BFM,Parry}).

In this paper, we will approach a problem of two of the authors and Males \cite{FMR} on {\it hook polynomials} which arose in their study of the connection between zeros of partition polynomials and crank statistics. In order to explain the problem, we require hook numbers of partitions. For each cell of the Young diagram of a partition $\lambda$, the {\it hook number} of this cell is the number of boxes that form the inverted $L$-shape whose corner resides at the cell in question. For instance, the hook numbers of the partition $\lambda = \lp 5, 2, 1 \rp$ are given as
\begin{figure}[ht] ${\hspace{0.4in} \begin{ytableau} 7&5&3&2&1 \\ 3&1 \\ 1 \end{ytableau}}$
\end{figure}

 We let $\mathcal H(\lambda)$ be the multiset of all hook numbers of $\lambda$ and the $t$-hook multiset as
\begin{align*}
 \mathcal H_t(\lambda) := \{ h \in \mathcal H(\lambda) : t \mid h \}.
\end{align*}
In order to keep track of the number of $t$-hooks of partitions, we use a formula for the generating function due to Han, which derives from his extensions of the Nekrasov--Okounkov hook lengths formula \cite{Han,NekOk}. To state this theorem and those that follow, we define for each $t \geq 1$ the generating functions
\begin{align*}
 H_t\lp w;q \rp := \sum_{n \geq 0} P_t(w;n) q^n := \sum_{\lambda} w^{\#\mathcal H_t(\lambda)} q^{|\lambda|} .
\end{align*}

The following representation for $H_t(w;q)$ is proven in \cite[Corollary 5.1]{Han}.

\begin{Theorem}
 We have for $t \geq 2$ that
 \begin{align*}
 H_t(w;q) = \prod_{n \geq 1} \frac{1}{(1-(wq^t)^n)^t}\prod_{n \geq 1} \frac{\big(1-q^{tn}\big)^t}{1-q^n}.
 \end{align*}
\end{Theorem}

Males and two of the authors \cite{FMR} have asked how the zeros of the polynomials $P_t(w,n)$ are distributed as $n \to \infty$. The data for $t \geq 6$ seems to indicate that these polynomials have isolated zeros outside the unit circle, and chaotic patterns in an annulus inside the unit circle (see Figures~\ref{HookZeros5mod7} and~\ref{HookZeros6mod7}). There also appears to be a difference in plots across the residue class of~${n \pmod{t}}$. We make no attempt here to describe a zero attractor, but we prove asymptotics for $P_t(w,n)$ that imply that there are no zeros outside a certain annulus, with the exception of isolated points coming from zeros of a theta function.

\begin{figure}[t]
\centering
\includegraphics[scale=.26]{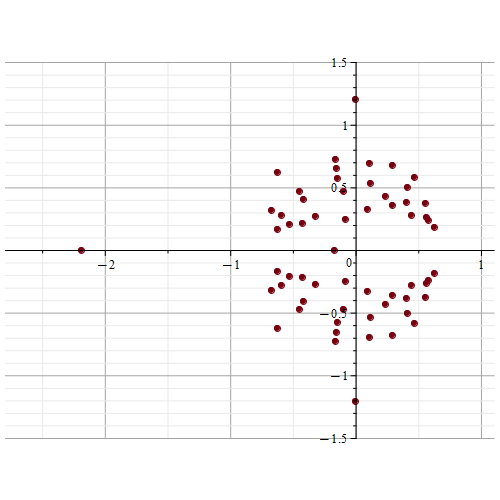}
 \includegraphics[scale=.26]{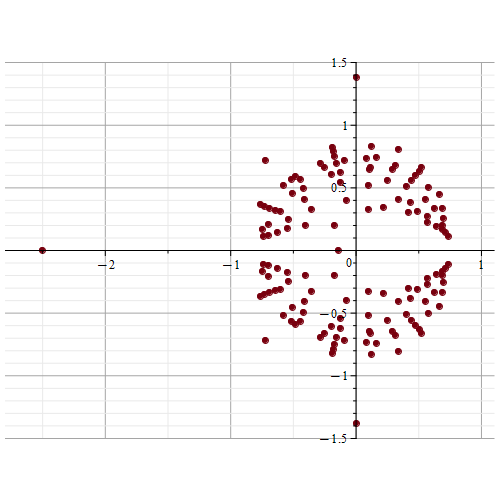}
 \includegraphics[scale=.26]{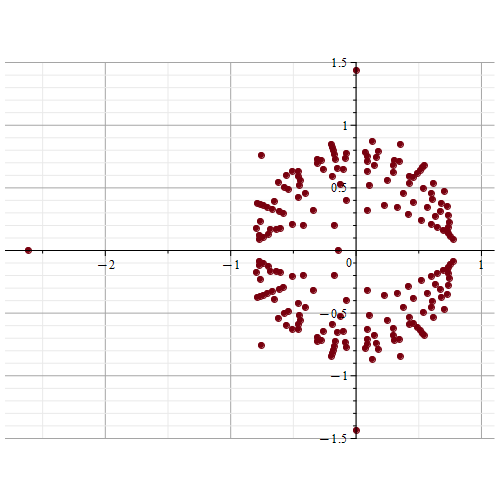}
\caption{Zeros of $P_7(w,n)$ for $n=425,985$ and $1405$.}
\label{HookZeros5mod7}
\end{figure}

\begin{figure}[t]
\centering
\includegraphics[scale=.25]{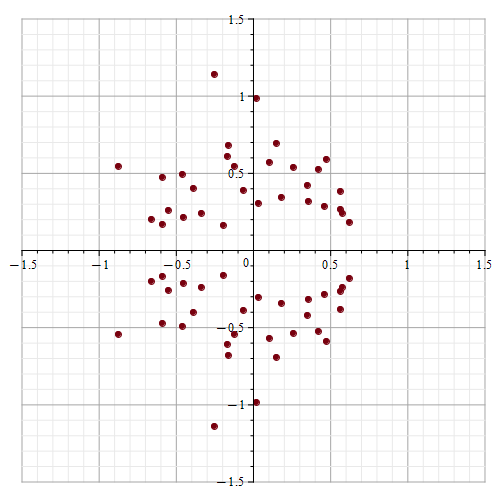} \hspace{.08in}
 \includegraphics[scale=.25]{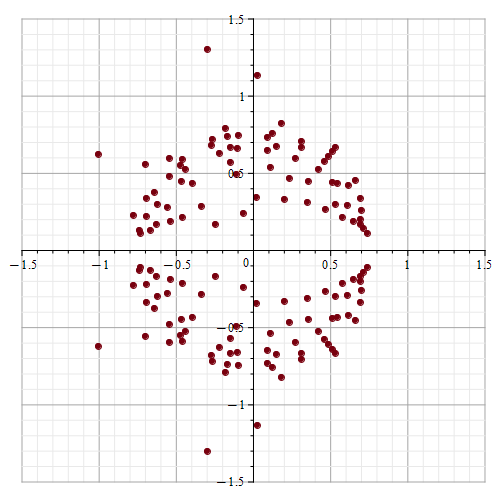}
 \includegraphics[scale=.25]{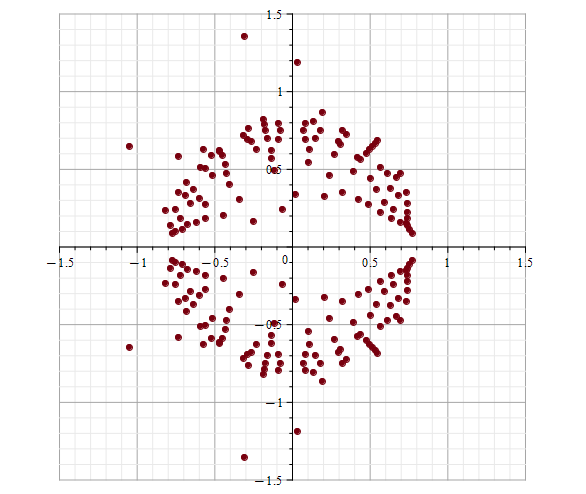}
\caption{Zeros of $P_7(w,n)$ for $n=426$, $986$ and $1406$.}
\label{HookZeros6mod7}
\end{figure}

 To state the asymptotics for $P_t(w,n)$ in the case $|w|>1$, we define, for $|z|<1$, the theta function
\begin{align*}
 \Theta_{\ell,t}(z):=\sum_{\substack{\boldsymbol{m}\in \mathbb{Z}^t \\ \boldsymbol{1} \cdot \boldsymbol{m}=0 \\ \boldsymbol{b}\cdot \boldsymbol{m} \equiv \ell \ ({\text{mod}\, {t}}) }} z^{\frac{1}{2}\|\boldsymbol{m}\|^2+\frac{1}{t}\boldsymbol{b}\cdot \boldsymbol{m}},
\end{align*} where $\boldsymbol{b}:=(0,1,\dots,t-1)$. For $\epsilon>0$, let
\begin{align*}
\mathcal{Z}_{\ell,t}(\epsilon):=\bigl\{z : |z|\geq 1+\epsilon, \, \Theta_{\ell,t}\big(z^{-1}\big)=0\bigr\}.
\end{align*}
Since $\Theta_{\ell,t}(z)$ is analytic and nonzero in $|z|<1$, it follows that $\#\mathcal{Z}_{\ell,t}(\epsilon)<\infty$ for all $\epsilon>0$.

\begin{Theorem} \label{|w|>1 Asymptotics}
 Let $t
 \geq 1$ and let $\epsilon>0$. Then uniformly for $|w|>1+\epsilon$ and $w$ not in an $\epsilon$-neighborhood of $\mathcal{Z}_{\ell,t}(\epsilon)$, we have, as $n \to \infty$ along the residue class $\ell \pmod{t}$,
 \begin{align*}
P_t(w,n)= {\rm e}^{\pi \sqrt{\frac{2n}{3}}}w^{\frac{n}{t}} \Theta_{\ell,t}\big(w^{-1}\big)\frac{1}{2^{\frac{5+t}{4}}3^{\frac{1+t}{4}}\pi^{\frac{3+t}{2}}n^{\frac{3+t}{4}}}(1+o(1)).
 \end{align*}
\end{Theorem}

 In the case $|w|<1$, our asymptotic results, as well as the regions in which they hold, depend on the function $A_t(w,n)$ defined in equation \eqref{A Definition}.

\begin{Theorem} \label{|w|<1 Asymptotics}
 Let $t \geq 6$. Let $w_0=w_0(t)$ be chosen, independent of $n$, so that $A_t(w,n) \neq 0$ for all $|w|\leq w_0$ $($see Section $\ref{S:smallw})$. Then uniformly for $|w|\leq w_0$, we have, as $n \to \infty$,
 \begin{align*}
 P_t(w,n)= \frac{(2\pi)^{\frac{t-1}{2}}A_t(w,n)}{t^{t/2}\Gamma\bigl(\frac{t-1}{2}\bigr)}n^{\frac{t-3}{2}} + O\big(n^{\frac{t-1}{4}}\big).
 \end{align*}
\end{Theorem}

 The following result controls the zero set of $P_t(w,n)$ and is an immediate corollary of the uniform growth of $P_t(w,n)$ in the respective regions of Theorems \ref{|w|>1 Asymptotics} and \ref{|w|<1 Asymptotics}.

\begin{Corollary}
 Let $t \geq 6$ and let $\ell \in \mathbb{N}$. For all $\epsilon>0$, there exists an $n_0=n_0(\epsilon,\ell,t)$ such that the zeros of $P_t(w,n)$ for $n \equiv \ell \pmod{t}$ and $n \geq n_0$ lie either in an $\epsilon$-neighborhood of $\mathcal{Z}_{\ell,t}(\epsilon)$ or in the annulus $w_0\leq |w|\leq 1+\epsilon$.
\end{Corollary}

\begin{Remark}
 It would be interesting to further describe asymptotics or the zero set of $P_t(w,n)$ in the region $w_0<|w|<1$. We also leave the optimization of $w_0$ and a deeper study of the function $A_t(w,n)$ as interesting follow-up problems.
\end{Remark}

The key technical difference in the hook polynomials $P_t(w,n)$ and the number-of-parts polynomials $Q_n(w)$ in \eqref{E:QnwDEF} studied by Boyer and Goh \cite{BG1,BG2} is that $w$ appears only in {\it some of} the factors in the product $H_t(w;q)$, rather than in all of them. Depending on whether $|w|<1$ or~${|w|>1}$, this changes the radius of convergence of $H_t(w;q)$ and leads to the two distinct cases of Theorems \ref{|w|>1 Asymptotics} and \ref{|w|<1 Asymptotics}. The general philosophy is analogous to Sokal's work on chromatic polynomials, where {\it phase transitions} determine zero attractors of a sequence of polynomials. In the case $|w|<1$, our proof of Theorem \ref{|w|<1 Asymptotics} is a straightforward adjustment of Anderson's proof of an asymptotic formula for the number of $t$-core partitions for $t \geq 6$ \cite{Anderson}. As such, we do not address the cases $t\in \{2,3,4,5\}$, where the asymptotics of $c_t(n)$ are less predictable; for example, $c_2(n)$ and $c_3(n)$ are known to be lacunary (see, e.g., \cite[pp.~333--334]{GranvilleOno}).

 In Section \ref{S: Asymptotics}, we prove Theorems \ref{|w|>1 Asymptotics} and \ref{|w|<1 Asymptotics}. In Section \ref{S: Conjectures}, we discuss some open problems and data on other polynomial families, especially in connection with deformations of Rogers--Ramanujan identities.

\section{Asymptotics for hook-length polynomials} \label{S: Asymptotics}

\subsection[The case |w|>1]{The case $\boldsymbol{|w|>1}$} 

Throughout this section, we use the abbreviated $q$-Pochhammer,
\begin{align*}
(a)_{\infty}:=(a;a)_{\infty}=\prod_{k \geq 1} \big(1-a^k\big).
\end{align*}
Note that by \cite[equation~(3.1)]{GKS}, we may rewrite
\begin{align*}
 \Theta_{\ell,t}(z)&{}=\sum_{\substack{\boldsymbol{m}\in \mathbb{Z}^t \\ \boldsymbol{1} \cdot \boldsymbol{m}=0 \\ \boldsymbol{b}\cdot \boldsymbol{m} \equiv \ell \ (\text{mod}\, {t}) }} \! z^{\frac{1}{2}\|\boldsymbol{m}\|^2+\frac{1}{t}\boldsymbol{b}\cdot \boldsymbol{m}}
 = (z)_{\infty}^t \sum_{\substack{m \geq 0 \\ m \equiv \ell \ (\text{mod}\, {t})}}\! p(m)z^{\frac{m}{t}}
=\frac{1}{t}\sum_{j=0}^{ t-1}{\rm e}^{-2\pi {\rm i} \frac{j\ell}{t}}\frac{(z)_{\infty}^t}{\big(z^{\frac{1}{t}}{\rm e}^{2\pi {\rm i} \frac{j}{t}}\big)_{\infty}}.
\end{align*}
It is in the latter form that $\Theta_{\ell,t}$ will appear in the proof of Theorem \ref{|w|>1 Asymptotics}. We will also make use of the version of Laplace's Method below that is sufficient for our purposes.
\begin{Theorem}[Laplace's method, see {\cite[Section 1.1.5]{Pinsky}}]\label{T:LaplaceMethod}
 Let $A,B\colon [a,b]\to \mathbb{C}$ be continuous functions. Let $x_0 \in (a,b)$ and suppose $\mathrm{Re}(B(x))<\mathrm{Re}(B(x_0))$ for all $x\in (a,b)$ with $x \neq x_0$. Suppose further that
 \begin{align*}
 \lim_{x\to x_0}\frac{B(x)-B(x_0)}{(x-x_0)^2} = -k\in \mathbb{C},
 \end{align*}
 with $\mathrm{Re}(k)>0.$ Then as $t \to \infty$,
 \begin{align*}
 \int^b_{a}A(x){\rm e}^{tB(x)}{\rm d}x = {\rm e}^{tB(x_0)}\biggl(A(x_0)\sqrt{\frac{\pi}{tk}}+o\biggl(\frac{1}{\sqrt{t}}\biggr)\biggr).
 \end{align*}
\end{Theorem}

We are now ready to prove Theorem \ref{|w|>1 Asymptotics}.

\begin{proof}[Proof of Theorem \ref{|w|>1 Asymptotics}] If $|w|>1$, then $P_t(w,q)$ is analytic in the region
 \begin{align*}
 \bigl|wq^t\bigr|<1 \ \iff\  |q|< |w|^{-\frac{1}{t}}<1.
 \end{align*}
 Now, let \smash{$x_n:= \frac{\pi}{\sqrt{6n}} \to 0^+$}. Choose a $t$-th root of $w$. Then
 \begin{align*}
 P_t(w,n)&{}=\frac{1}{2\pi {\rm i}} \int_{\substack{q=w^{-\frac{1}{t}}{\rm e}^{-x_n+2\pi {\rm i} \theta} \\ |\theta|\leq \frac{1}{2}}} \frac{H_t(w;q)}{q^{n+1}} \,{\rm d}q \\
 &{}={\rm e}^{nx_n}w^{\frac{n}{t}}\int_{|\theta|\leq \frac{1}{2}} H_t\big(w;w^{-\frac{1}{t}}{\rm e}^{-x_n+2\pi {\rm i} \theta}\big){\rm e}^{-2\pi {\rm i} n \theta} {\rm d}\theta \\
 &{}={\rm e}^{nx_n}w^{\frac{n}{t}}\int_{|\theta|\leq \frac{1}{2}} \big({\rm e}^{-tx_n+2\pi {\rm i} t \theta}\big)_{\infty}^{-t}\frac{\big(w^{-1}{\rm e}^{-tx_n+2\pi {\rm i} t \theta}\big)_{\infty}^t}{\big(w^{-\frac{1}{t}}{\rm e}^{-x_n+2\pi {\rm i} \theta}\big)_{\infty}}{\rm e}^{-2\pi {\rm i} n \theta} {\rm d}\theta.
 \end{align*}
 Now set \smash{$\theta\mapsto \frac{1}{t}(\theta+j)$} for $0\leq j \leq t-1$. Then we obtain
 \begin{align*}
 P_t(w,n)
 =\frac{{\rm e}^{nx_n}w^{\frac{n}{t}}}{t}\sum_{j=0}^{ t-1}{\rm e}^{-2\pi {\rm i} \frac{jn}{t}}\int_{|\theta|\leq \frac{1}{2}} \big({\rm e}^{-tx_n+2\pi {\rm i} \theta}\big)_{\infty}^{-t}\frac{\big(w^{-1}{\rm e}^{-tx_n+2\pi {\rm i} \theta}\big)_{\infty}^t}{\big(w^{-\frac{1}{t}}{\rm e}^{2\pi {\rm i} \frac{j}{t}-x_n+2\pi {\rm i} \frac{\theta}{t}}\big)_{\infty}}{\rm e}^{-2\pi {\rm i} \frac{n}{t} \theta} {\rm d}\theta.
 \end{align*}
 Let $\mathcal{F}_N$ be the Farey sequence of order $N$ (to be chosen later in the proof) and $\theta_{h,k}'$, $\theta_{h,k}''$ the respective mediants of three consecutive fractions $\frac{h'}{k'},\frac{h}{k},\frac{h''}{k''} \in \mathcal{F}_N$. Then
 \begin{align*}
 P_t(w,n)
 &{}=\frac{{\rm e}^{nx_n}w^{\frac{n}{t}}}{t}\sum_{j=0}^{ t-1}{\rm e}^{-2\pi {\rm i} \frac{jn}{t}} \sum_{\frac{h}{k} \in \mathcal{F}_N} {\rm e}^{-2\pi {\rm i} \frac{hn}{kt}} \\
 &\hphantom{=\frac{{\rm e}^{nx_n}w^{\frac{n}{t}}}{t}\sum_{j=0}^{ t-1}}\,{} \times \int_{-\theta_{h,k}'}^{\theta_{h,k}''} \!\big({\rm e}^{-tx_n+2\pi {\rm i} \theta+2\pi {\rm i} \frac{h}{k}}\big)_{\infty}^{-t}\frac{\big(w^{-1}{\rm e}^{-tx_n+2\pi {\rm i} \theta+2\pi {\rm i} \frac{h}{k}}\big)_{\infty}^t}{(w^{-\frac{1}{t}}{\rm e}^{2\pi {\rm i} \frac{j}{t}-x_n+2\pi {\rm i} \frac{\theta}{t}+2\pi {\rm i} \frac{h}{kt}})_{\infty}}{\rm e}^{-2\pi {\rm i} \frac{n}{t} \theta} {\rm d}\theta.
 \end{align*}
 Since the middle factor of the integrand is analytic for $|q|=1$, we have, as $x_n \to 0^+$,
 \begin{align}
 \frac{\big(w^{-1}{\rm e}^{-tx_n+2\pi {\rm i} \theta+2\pi {\rm i} \frac{h}{k}}\big)_{\infty}^t}{\big(w^{-\frac{1}{t}}{\rm e}^{2\pi {\rm i} \frac{j}{t}-x_n+2\pi {\rm i} \frac{\theta}{t}+2\pi {\rm i} \frac{h}{kt}}\big)_{\infty}} =\frac{\big(w^{-1}{\rm e}^{2\pi {\rm i} \frac{h}{k}}\big)_{\infty}^t}{\big(w^{-\frac{1}{t}}{\rm e}^{2\pi {\rm i} \frac{j}{t}+2\pi {\rm i} \frac{h}{kt}}\big)_{\infty}}(1+x_n E_{j,h,k,t,n,w}(\theta)), \label{E:Ejhktw}
 \end{align}
 where $E_{j,h,k,t,n,w}(\theta)$ is a continuous function of $\theta$, uniformly bounded with respect to~$j$,~$h$,~$k$,~$n$ and for $|w|\geq 1+\epsilon$. Set $N:=\lfloor \sqrt{n} \rfloor$.
 Here, $\theta_{h,k}',\theta_{h,k}'' \asymp \frac{1}{kN} \to 0$ \cite[p.~75]{Andrews}. Thus, using the continuity of the above function of $\theta$,
 \begin{align*}
 P_t(w,n)
 &{}=\frac{{\rm e}^{nx_n}w^{\frac{n}{t}}}{t}\sum_{j=0}^{ t-1}{\rm e}^{-2\pi {\rm i} \frac{jn}{t}} \sum_{\frac{h}{k} \in \mathcal{F}_N} {\rm e}^{-2\pi {\rm i} \frac{hn}{kt}}\frac{\big(w^{-1}{\rm e}^{2\pi {\rm i} \frac{h}{k}}\big)_{\infty}^t}{\big(w^{-\frac{1}{t}}{\rm e}^{2\pi {\rm i} \frac{j}{t}+2\pi {\rm i} \frac{h}{kt}}\big)_{\infty}} \\[-1mm]
 &\hphantom{=\frac{{\rm e}^{nx_n}w^{\frac{n}{t}}}{t}\sum_{j=0}^{ t-1}}\,{} \times \int_{-\theta_{h,k}'}^{\theta_{h,k}''} \big({\rm e}^{-tx_n+2\pi {\rm i} \theta+2\pi {\rm i} \frac{h}{k}}\big)_{\infty}^{-t}{\rm e}^{-2\pi {\rm i} \frac{n}{t} \theta}(1+x_n E_{j,h,k,n,t,w}(\theta)) \,{\rm d}\theta \\[-1mm]
 &{}={\rm e}^{nx_n}w^{\frac{n}{t}} \sum_{\frac{h}{k} \in \mathcal{F}_N} {\rm e}^{-2\pi {\rm i} \frac{hn}{kt}}\Theta_{n,t}\big(w^{-1}{\rm e}^{2\pi {\rm i} \frac{h}{k}}\big) \\[-1mm]
 &\hphantom{={\rm e}^{nx_n}w^{\frac{n}{t}} \sum_{\frac{h}{k} \in \mathcal{F}_N}}\,{} \times \int_{-\theta_{h,k}'}^{\theta_{h,k}''} \big({\rm e}^{-tx_n+2\pi {\rm i} \theta+2\pi {\rm i} \frac{h}{k}}\big)_{\infty}^{-t}{\rm e}^{-2\pi {\rm i} \frac{n}{t} \theta}(1+x_n E_{j,h,k,n,t,w}(\theta)) \,{\rm d}\theta.
 \end{align*}

 Note that $\Theta_{n,t}\big(w^{-1}{\rm e}^{2\pi {\rm i} \frac{h}{k}}\big)=\Theta_{\ell,t}\big(w^{-1}{\rm e}^{2\pi {\rm i} \frac{h}{k}}\big)$ is uniformly bounded for all $\frac{h}{k} \in \mathcal{F}_N$ and all $|w|>1+\epsilon$, so does not contribute to asymptotic growth. Now, by \cite[equation~(5.2.2)]{Andrews}, with
 \begin{align*}
 2\pi {\rm i} \frac{h+{\rm i}z}{k}=-tx_n+2\pi {\rm i} \theta+2\pi {\rm i} \frac{h}{k} \ \implies\  z=\frac{ktx_n}{2\pi}-k{\rm i}\theta,
 \end{align*}
 and $hh' \equiv -1 \pmod{k}$ and $\omega_{h,k}={\rm e}^{\pi {\rm i} s(h,k)}$, where $s(h,k)$ is the Dedekind sum as in \cite[equation~(5.2.4)]{Andrews},
\begin{align*}
 &\big({\rm e}^{-tx_n+2\pi {\rm i} \theta+2\pi {\rm i} \frac{h}{k}}\big)_{\infty}^{-t} \\[-1mm]
 &{}= \omega_{h,k}^t \biggl(\frac{ktx_n}{2\pi}-k{\rm i}\theta\biggr)^{\frac{t}{2}}{\rm e}^{\frac{\pi t}{12k(\frac{ktx_n}{2\pi}-k{\rm i}\theta)}-\pi t\frac{\frac{ktx_n}{2\pi}-k{\rm i}\theta}{12k}}\Biggl(1+O\Biggl( {\rm e}^{2\pi {\rm i}\frac{h'+\frac{{\rm i}}{\frac{ktx_n}{2\pi}-k{\rm i}\theta}}{k}}\Biggr) \Biggr) \\[-1mm]
 &{}= \omega_{h,k}^t k^{\frac{t}{2}} \biggl(\frac{tx_n}{2\pi}-{\rm i}\theta\biggr)^{\frac{t}{2}}{\rm e}^{\frac{\pi t}{12}\big(\frac{1}{k^2(\frac{tx_n}{2\pi}-{\rm i}\theta)}-(\frac{tx_n}{2\pi}-{\rm i}\theta)\big)}\Biggl(1+\Biggl( {\rm e}^{- \frac{2\pi}{1+(\frac{2\pi\theta}{tx_n})^2}}\Biggr) \Biggr) \\[-1mm]
 &{}= \omega_{h,k}^t k^{\frac{t}{2}} \biggl(\frac{tx_n}{2\pi}-{\rm i}\theta\biggr)^{\frac{t}{2}}{\rm e}^{\frac{\pi t}{12k^2(\frac{tx_n}{2\pi}-{\rm i}\theta)}}(1+O(x_n)) \biggl(1+O\biggl(\frac{1}{kN}\biggr) \biggr) \Biggl(1+O\Biggl( \frac{1}{1+\frac{1}{(tx_nkN)^2}}\Biggr) \Biggr).
\end{align*}
Thus,
 \begin{align}
 P_t(w,n) &{}={\rm e}^{nx_n}w^{\frac{n}{t}} \sum_{\frac{h}{k} \in \mathcal{F}_{\lfloor \sqrt{n} \rfloor}} \omega_{h,k}^tk^{\frac{t}{2}}{\rm e}^{-2\pi {\rm i} \frac{hn}{kt}}\Theta_{\ell,t}\big(w^{-1}{\rm e}^{2\pi {\rm i} \frac{h}{k}}\big)(1+o(1)) \label{E:CircleMethodSeries}\\[-1mm]
 &\hphantom{={\rm e}^{nx_n}w^{\frac{n}{t}} \sum_{\frac{h}{k} \in \mathcal{F}_{\lfloor \sqrt{n} \rfloor}}}{}
 \times \int_{-\theta_{h,k}'}^{\theta_{h,k}''} \! \biggl(\frac{tx_n}{2\pi}-{\rm i}\theta\biggr)^{\frac{t}{2}}\!{\rm e}^{\frac{\pi t}{12k^2(\frac{tx_n}{2\pi}-{\rm i}\theta)}-2\pi {\rm i} \frac{n}{t} \theta}\! (1+x_n E_{j,h,k,t,n,w}(\theta))\, {\rm d}\theta, \nonumber
 \end{align}
and we are left to bound the integrals. Choose $N:=\lfloor \sqrt{n} \rfloor$, so that for $k \geq 2$,{\samepage
\begin{align*}
 & \biggl|
 \int_{-\theta_{h,k}'}^{\theta_{h,k}''} \biggl(\frac{tx_n}{2\pi}-{\rm i}\theta\biggr)^{\frac{t}{2}}{\rm e}^{\frac{\pi t}{12k^2(\frac{tx_n}{2\pi}-{\rm i}\theta)}-2\pi {\rm i} \frac{n}{t} \theta} (1+x_n E_{j,h,k,n,t,w}(\theta)) \,{\rm d}\theta \biggr| \\[-1mm]
 &\qquad{} \ll \int_{-\theta_{h,k}'}^{\theta_{h,k}''} \biggl|\frac{tx_n}{2\pi}-{\rm i}\theta \biggr|^{\frac{t}{2}} {\rm e}^{\frac{\pi t \cdot \frac{tx_n}{2\pi}}{12k^2(\frac{t^2x_n^2}{4\pi^2}+\theta^2)}} {\rm d}\theta
 \ll \frac{1}{kN}{\rm e}^{\frac{\pi^2}{6k^2x_n}}
 =o\Bigl({\rm e}^{\frac{\pi^2}{6x_n}}\Bigr),
\end{align*}
where we used the boundedness of $E_{j,h,k,n,t,w}(\theta)$.}

For $k=1$, we have $\theta_{0,1}'=\theta_{0,1}''=\frac{1}{N+1}$ and we apply Theorem \ref{T:LaplaceMethod}.
We first need to transform the integral to be of the correct form. We first set $\theta \mapsto \theta \frac{tx_n}{2\pi}$. Thus,
\begin{align}
 &\int_{-\frac{1}{N+1}}^{\frac{1}{N+1}} \biggl(\frac{tx_n}{2\pi}-{\rm i}\theta\biggr)^{\frac{t}{2}}{\rm e}^{\frac{\pi t}{12(\frac{tx_n}{2\pi}-{\rm i}\theta)}-2\pi {\rm i} \frac{n}{t} \theta} (1+x_n E_{j,0,1,n,t,w}(\theta)) \,{\rm d}\theta \nonumber \\
 &\qquad{}=\frac{tx_n}{2\pi}\int_{-\frac{2\pi}{tx_n(N+1)}}^{\frac{2\pi}{tx_n(N+1)}} \biggl(\frac{tx_n}{2\pi}-\frac{{\rm i}tx_n\theta}{2\pi}\biggr)^{\frac{t}{2}}{\rm e}^{\frac{\pi}{12(\frac{x_n}{2\pi}-\frac{{\rm i}x_n\theta}{2\pi})}- {\rm i} nx_n \theta} \biggl(1+x_n E_{j,0,1,n,t,w}\biggl(\frac{tx_n\theta}{2\pi}\biggr)\biggr) {\rm d}\theta \nonumber \\
 &\qquad{}=\frac{t^{\frac{t+2}{2}}x_n^{\frac{t+2}{2}}}{(2\pi)^{\frac{t+2}{2}}}\int_{-\frac{2\pi}{tx_n(N+1)}}^{\frac{2\pi}{tx_n(N+1)}} (1-{\rm i}\theta)^{\frac{t}{2}}{\rm e}^{\frac{\pi^2}{6x_n}( \frac{1}{1-{\rm i}\theta}- {\rm i} \theta)} \biggl(1+x_n E_{j,0,1,n,t,w}\biggl(\frac{tx_n\theta}{2\pi}\biggr)\biggr) {\rm d}\theta. \label{E:MainRewrite}
\end{align}
Now, as \smash{$x_n, \frac{1}{N+1} \asymp \frac{1}{\sqrt{n}}$}, it follows that the bounds of the integral above are bounded away from~0 and~$\infty$ and are of the form needed for Theorem \ref{T:LaplaceMethod}. We would like to apply Theorem \ref{T:LaplaceMethod} with $x\mapsto \theta$, $x_0\mapsto 0$, \smash{$A(\theta)=(1-{\rm i}\theta)^{\frac{1}{2}}$} and $B(\theta)=\frac{1}{1-{\rm i}\theta}- {\rm i} \theta$. Here,
\begin{align*}
 \mathrm{Re} \lp B(\theta) \rp = \frac{1}{1+\theta^2}< 1
\end{align*}
for $\theta \neq 0$ and Taylor expanding gives
\begin{align*}
 B(\theta)=1-\theta^2+O\big(\theta^3\big).
\end{align*}
First we bound the error term, using the boundedness of $E_{j,0,1,n,t,w}$, as
\begin{equation}\label{E:MainTerm}
\biggl|\int_{-\frac{2\pi}{tx_n(N+1)}}^{\frac{2\pi}{tx_n(N+1)}} (1-{\rm i}\theta)^{\frac{t}{2}}{\rm e}^{\frac{\pi^2}{6x_n}( \frac{1}{1-{\rm i}\theta}- {\rm i}\theta)} x_n E_{j,0,1,n,t,w}\biggl(\frac{tx_n\theta}{2\pi}\biggr) {\rm d}\theta \biggr| \ll {\rm e}^{\frac{\pi^2}{6x_n}}x_n.
\end{equation}
For the main term, Theorem \ref{T:LaplaceMethod} gives
\begin{equation}\label{E:MainError}
\int_{-\frac{2\pi}{tx_n(N+1)}}^{\frac{2\pi}{tx_n(N+1)}} (1-{\rm i}\theta)^{\frac{t}{2}}{\rm e}^{\frac{\pi^2}{6x_n}( \frac{1}{1-{\rm i}\theta}- {\rm i}\theta)}={\rm e}^{\frac{\pi^2}{6x_n}}\Biggl(\sqrt{\frac{\pi6x_n}{\pi^2}}+o\big( \sqrt{x_n}\big)\Biggr),
\end{equation}
hence \eqref{E:MainRewrite}--\eqref{E:MainError} imply
\begin{align}
 \int_{-\frac{1}{N+1}}^{\frac{1}{N+1}} \biggl(\frac{tx_n}{2\pi}-{\rm i}\theta\biggr)^{\frac{t}{2}}{\rm e}^{\frac{\pi t}{12(\frac{tx_n}{2\pi}-{\rm i}\theta)}-2\pi {\rm i} \frac{n}{t} \theta} {\rm d}\theta
 &{}=\frac{t^{\frac{2+t}{2}}x_n^{\frac{2+t}{2}}}{(2\pi)^{\frac{2+t}{2}}}{\rm e}^{\frac{\pi^2}{6x_n}}\Biggl(\sqrt{\frac{\pi6x_n}{\pi^2}}+o( 1) \Biggr) \nonumber \\
 &{}=\frac{1}{2^{\frac{5+t}{4}}3^{\frac{1+t}{4}}\pi^{\frac{3+t}{2}}n^{\frac{3+t}{4}}}{\rm e}^{\frac{\pi\sqrt{n}}{\sqrt{6}}}(1+o(1)). \label{E:wlargemainterm}
\end{align}
Since we assume that $w^{-1}$ is not in an $\epsilon$-neighborhood of $\mathcal{Z}_{\ell,t}(\epsilon)$, it follows that the main term in \eqref{E:CircleMethodSeries} does not vanish and occurs for $k=1$. Combining with \eqref{E:wlargemainterm} and using $\omega_{0,1}=1$, the theorem follows.
\end{proof}

\subsection[The case |w|<1]{The case $\boldsymbol{|w|<1}$} \label{S:smallw}

Let $\omega_{h,k}$ be as in the proof of Theorem \ref{|w|<1 Asymptotics}. Given $t \in \mathbb{N}$, following \cite[Theorem 5]{Anderson}, let
\begin{align*}
 \omega_{h,k}':=\frac{\omega_{h,k}}{\omega_{\frac{th}{(t,k)},\frac{k}{(t,k)}}^t}.
\end{align*}
We require the following twist of the function $A_t(n)$ from \cite{Anderson}:
\begin{align} \label{A Definition}
 A_t(w,n):=\sum_{\substack{k \geq 1 \\ \gcd(k,t)=1}} \frac{1}{k^{\frac{t-1}{2}}}\sum_{\substack{0\leq h < k \\ \gcd(h,k)=1}}{\rm e}^{-2\pi {\rm i} \frac{hn}{k}}\omega'_{h,k}\cdot \big(w{\rm e}^{2\pi {\rm i} \frac{ht}{k}};w{\rm e}^{2\pi {\rm i} \frac{ht}{k}}\big)_{\infty}^{-t}.
\end{align} Note that $A_t(0,n)=A_t(n)$.
 This series converges absolutely and uniformly in $w$ for $|w|$ sufficiently small; indeed, we can trivially bound the series by $(|w|)_{\infty}^{-t}\zeta\big(\frac{3}{2}\big)$. Therefore, we have $\lim_{w \to 0} A_t(w,n)=A_t(n)$, and it follows by the following proposition that $A_t(w,n) \neq 0$ for $|w|$ sufficiently small.

\begin{Proposition}[{\cite[Proposition 6]{Anderson}}] If $t \geq 6$, then $A_t(n)$ is a nonzero real number such that
\begin{align*}
0.05 < A_t(n) < 2.62.
\end{align*}
\end{Proposition}

Our proof of Theorem \ref{|w|<1 Asymptotics} closely follows the proof in \cite{Anderson} of an asymptotic formula for the number of $t$-core partitions in \cite{Anderson}. For further developments in these asymptotics, see \cite{LP, Tyler}.

\begin{proof}[Proof of Theorem \ref{|w|<1 Asymptotics}]

In this case, $H_t(w,q)$ is analytic in the region $|q|<1$, since also ${\bigl|wq^t\bigr|<1}$. We closely follow \cite{Anderson}. Exactly as in \cite[p.~2598]{Anderson}, we have
\begin{align*}
P_t(w;n) &={\rm i}\sum_{\substack{0\leq h < k \leq N \\ \gcd(h,k)=1}}{\rm e}^{-2\pi {\rm i} \frac{hn}{k}}\int_{N^{-2}+{\rm i}\theta_{h,k}'}^{N^{-2}-{\rm i}\theta_{h,k}''}H_t\biggl(w;\exp\biggl( \frac{2\pi {\rm i} h}{k}-2\pi z\biggr)\biggr) {\rm e}^{2\pi {\rm i} n z} {\rm d}z \\
&={\rm i}\sum_{\substack{0\leq h < k \leq N \\ \gcd(h,k)=1}}{\rm e}^{-2\pi {\rm i} \frac{hn}{k}}\int_{N^{-2}+{\rm i}\theta_{h,k}'}^{N^{-2}-{\rm i}\theta_{h,k}''}\bigl(w {\rm e}^{\frac{2\pi {\rm i} h}{k}-2\pi z}\bigr)_{\infty}^{-t}F_t\biggl(\exp\biggl( \frac{2\pi {\rm i} h}{k}-2\pi z\biggr)\biggr) {\rm e}^{2\pi {\rm i} n z}{\rm d}z.
\end{align*} Set $N:=\lfloor \sqrt{n} \rfloor$.
 Here, $\theta_{h,k}',\theta_{h,k}'' \asymp \frac{1}{kN} \to 0$ \cite[p.~75]{Andrews}. Thus, ${\rm e}^{\frac{2\pi {\rm i} h}{k}-2\pi z} \to {\rm e}^{\frac{2\pi {\rm i} h}{k}}$ for $z$ in the range of integration. Hence, we have
 \begin{align*}
P_t(w;n)&{}={\rm i}\sum_{\substack{0\leq h < k \leq N \\ \gcd(h,k)=1}}{\rm e}^{-2\pi {\rm i} \frac{hn}{k}}\big(w{\rm e}^{\frac{2\pi {\rm i} h}{k}}\big)_{\infty}^{-t} \\
 & \hphantom{={\rm i}\sum_{\substack{0\leq h < k \leq N \\ \gcd(h,k)=1}}}\,{} \times \int_{N^{-2}+{\rm i}\theta_{h,k}'}^{N^{-2}-{\rm i}\theta_{h,k}''}F_t\biggl(\exp\biggl( \frac{2\pi {\rm i} h}{k}-2\pi z\biggr)\biggr) {\rm e}^{2\pi {\rm i} n z} (1+zE_{h,k,n,t,w}(z)) \,{\rm d}z,
\end{align*}
 where, as in \eqref{E:Ejhktw}, $E_{h,k,n,t,w}(z)$ is a continuous function of $z$, uniformly bounded with respect to~$h$,~$k$ and~$n$ for $|w|\leq w_0$; since $A_t(w,n) \neq 0$ and $z \to 0$, the contribution from this function goes into the error term. The rest of the proof follows \cite{Anderson}; after a transformation to $F_t$, we have
 \begin{align*}
 P_t(w,n)=M+E_1+E_2,
 \end{align*}
 where $M_1$, $E_1$, $E_2$ are as in \cite[p.~2599]{Anderson} with the extra factor \smash{$\big(w{\rm e}^{\frac{2\pi {\rm i} h}{k}}\big)_{\infty}^{-t}$} in the outer sum. Since this factor is uniformly bounded and $E_1$ and $E_2$ are bounded term-wise by inserting absolute values, we can bound $E_1$ and $E_2$ in our case in the same way and get
 \begin{align*}
|E_1|,|E_2|=O\big(n^{\frac{t-1}{4}}\big).
 \end{align*}
 The integral in $M$ is evaluated as in \cite[pp.~2602--2603]{Anderson} and completed to a full sum, from~${k \leq N}$ to~${k \geq 1}$, as in \cite[p.~2603, middle]{Anderson}, again using the uniform boundedness of \smash{\raisebox{-0.5pt}{$\big(w{\rm e}^{\frac{2\pi {\rm i} h}{k}}\big)_{\infty}^{-t}$}}.
\end{proof}

\section{Other polynomials} \label{S: Conjectures}

Following the introduction, we pose as a research area the detailed study of the zeros of partition polynomials. These can come from natural two-variable generalizations of $q$-series arising in partition theory, as well as those two-variable series which track natural statistics on partitions (of course, often these two situations coincide). To give an idea of the richness of the potential of this area of study, we give a couple examples where the zero attractors appear to give very different behaviors.

\subsection{Polynomials from products}

As previously mentioned, an important paper of Parry \cite{Parry} has computed asymptotics for many products in the case $|w|<1$; more specifically, Parry's results deal with very general products of the form
\begin{align*}
 \prod_{m \geq 1} \dfrac{1}{\lp 1 - w q^m \rp^{a_m}}.
\end{align*}
Using the strategy of Meinardus (see \cite[Section 6]{Andrews} and \cite{Meinardus}) suitably modified to permit the introduction of a second variable, Parry shows how analytic properties of the Dirichlet series
\begin{align*}
 \sum_{m \geq 1} \dfrac{{\rm e}^{\frac{2\pi {\rm i} h m}{k}} a_m}{m^s}
\end{align*}
produce asymptotic estimates for the coefficients of the related generating function in the case~${|w|<1}$. In such cases, the authors expect that asymptotics for $|w|>1$ follow using methods very much like those we have used, namely the saddle point method or circle method. It would be interesting to determine which kinds of zero attractors arise from natural partition generating functions that arise as products, including natural families of eta quotients and certain restricted partition functions such as partitions into powers (which are not connected to modular forms), the generating function $\lp -wq;q \rp_\infty$ which tracks the parts of partitions into distinct parts, or the D'Arcais polynomials which arise from the product $(q;q)_{\infty}^{-w}$ \cite{Han,HN19,HN21,HN23,HNW18,Kostant}.

\subsection{Polynomials from Rogers--Ramanujan series}

The Rogers--Ramanujan identities
\begin{align*}
 \sum_{n \geq 0} \dfrac{q^{n^2}}{\lp q;q \rp_n} = \dfrac{1}{\big( q; q^5 \big)_\infty \big( q^4; q^5 \big)_\infty} \qquad \text{and} \qquad \sum_{n \geq 0} \dfrac{q^{n^2 + n}}{\lp q; q \rp_n} = \dfrac{1}{\big( q^2; q^5 \big)_\infty \big( q^3; q^5 \big)_\infty}
\end{align*}
are among the central results in partition theory and in $q$-series, with far-reaching applications in many fields (see, for instance, \cite{Sills}). Zero attractors for coefficient polynomials associated to natural $w$-generalizations of the product sides of these identities (i.e., those that track the number of parts) can be analyzed using the framework of Parry discussed in the previous section. It would be natural to ask similar questions about the natural $w$-generalizations of the $q$-hypergeometric sides (which we note are distinct from the $w$-generalizations of the products). For instance, what are the behaviours of the polynomials arising from
\begin{align*}
 \sum_{n \geq 0} \dfrac{q^{n^2} w^n}{\lp q;q \rp_n} \qquad \text{ or } \qquad \sum_{n \geq 0} \dfrac{q^{n^2 + n} w^n}{\lp q;q \rp_n}?
\end{align*}
We also note that such functions have appeared already in physics \cite{Zhu}, and bilateral versions of these series were studied in \cite{Schlosser}. More broadly, we ask about the polynomials~$ p_{a,b}(w;n)$ defined~by
\begin{align*}
 P_{a,b}(w;q):= \sum_{n \geq 0} p_{a,b}(w;n) q^n := \sum_{n \geq 0} \dfrac{q^{an^2 + bn} w^n}{\lp q;q \rp_n}.
\end{align*}
There are of course other ways that $w$ may be inserted into $q$-hypergeometric series, which we leave to future researchers. We propose the following problem.

\begin{Problem}
 Classify and establish as a function of $a$ and $b$ the zero attractors of polynomial families $\{ p_{a,b}(w;n) \}_{n \geq 0}$.
\end{Problem}

In data we computed, we observed several separate behaviors in the zeros (not necessarily exhaustive or comprehensive in scope), which we outline below in four loosely titled subsections. In particular, we offer a claim surrounding the Rogers--Ramanujan functions.

\subsubsection{Asymptotically real} To illustrate this behavior, we consider the examples with $(a,b)=(1,0)$ and $(a,b)=(1,1)$, which correspond to the sum-sides of the first and second Rogers--Ramanujan identities, respectively;~i.e.,
\begin{align*}
&p_{1,0}(w;n)=
 \sum_m p(n : m \text{ superdistinct parts}) w^m, \\
&p_{1,1}(w;n)=
 \sum_m p(n : m \text{ superdistinct parts, parts $\geq 2$}) w^m.
\end{align*} Recall that a partition has {\it superdistinct} parts, if the difference between parts is at least 2. For~${b\in \{0,1\}}$, and $n\in \mathbb N$, we define the Rogers--Ramanujan zero sets by
\begin{align*} Z_{1,b}(n):=\{w \in \mathbb C : p_{1,b}(w;n)=0\}.\end{align*} We first observe that all zeros in these Rogers--Ramanujan zero sets appear to lie in the half-plane $\operatorname{Re}(w)\leq 0$, with many on the negative real axis. Moreover, as $n$ increases, the zeros appear to asymptotically approach this axis.
\begin{Claim} Let $b\in \{0,1\}$ and $n\in \mathbb N.$
 If $w_{1,b}(n) \in Z_{1,b}(n)$ then $\operatorname{Re}(w_{1,b}(n)) \leq 0.$
Moreover, $\lim_{n\to\infty}\mathrm{Im}(w_{1,b}(n)) =0.$
\end{Claim}

We illustrate this phenomenon with $(a,b)=(1,0)$ (the first Rogers--Ramanujan function), with $n=1000$, $n=5000$, $n=10000$ in Figure \ref{fig_rp1}, which plots the zeros of $p_{1,0}(w;n)$ with restricted real parts for visibility. Plots for $(a,b)=(1,1)$ corresponding to the second Rogers--Ramanujan function are similar (and omitted for brevity). For reference, $p_{1,0}(w;1000)$ is of degree $31$ and the zero which is approximately $-4936.858637$ (not visible in Figure \ref{fig_rp1}) is the one of maximum modulus; $p_{1,0}(w;5000)$ is of degree $70$ and the zero which is approximately $-485377.8433$ is the one of maximum modulus; $p_{1,0}(w;10000)$ is of degree $100$ and the zero which is approximately $-3.644618597\times 10^{12}$ is the one of maximum modulus.

\begin{figure}[ht]
\centering
 \includegraphics[scale=.21]{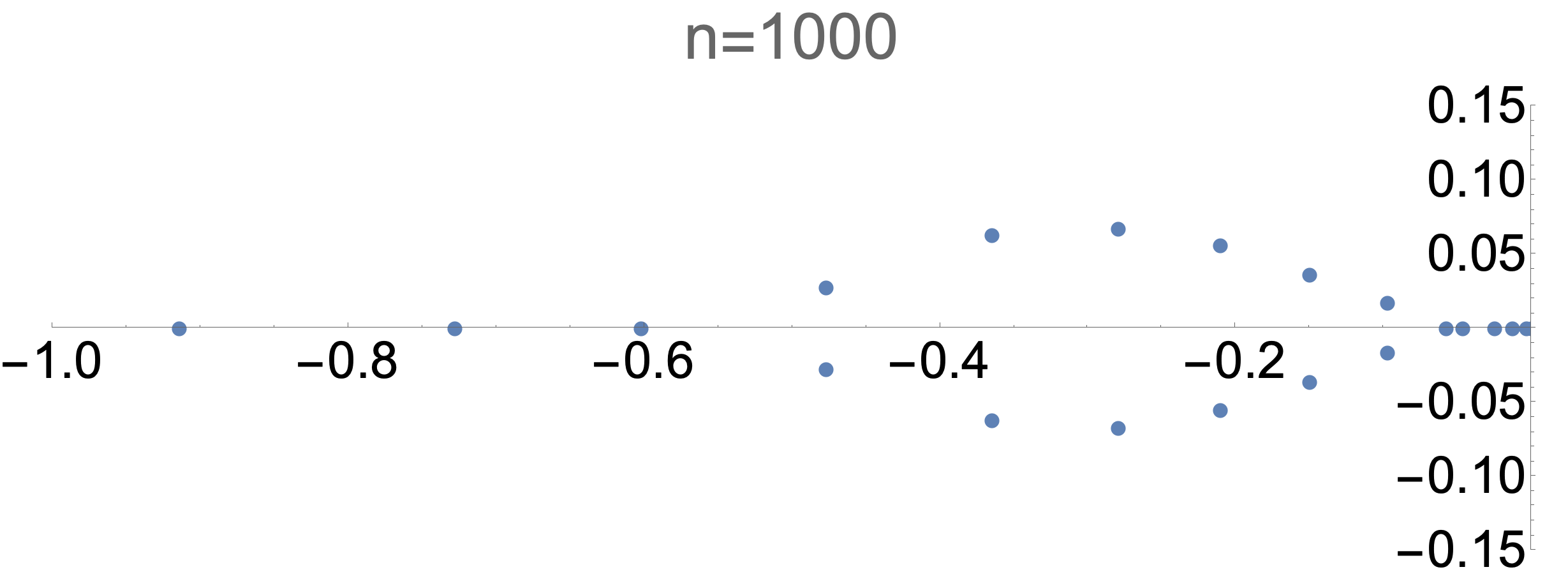} \medskip \\
\includegraphics[scale=.21]{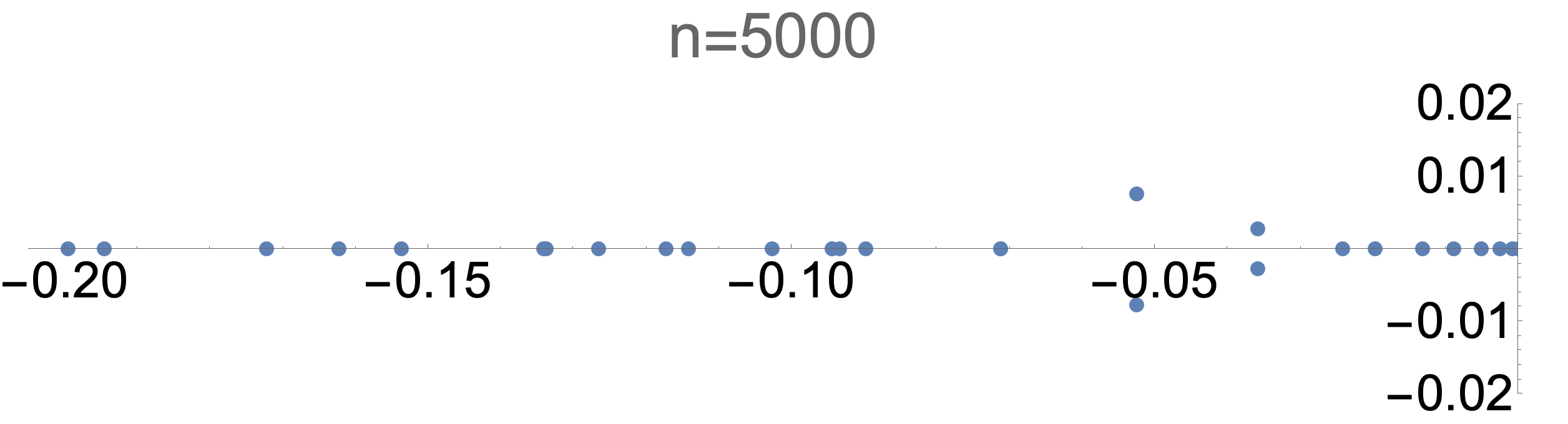} \medskip
 \\
\includegraphics[scale=.21]{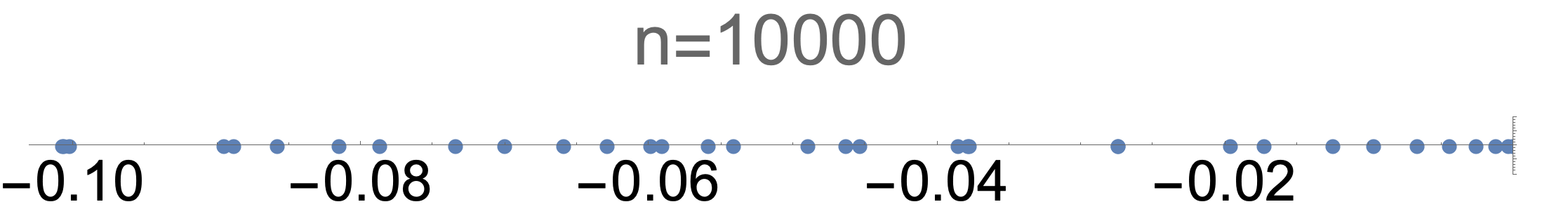}
\caption{Zeros of $p_{1,0}(w;n)$ with restricted real parts, for $n=1000, 5000$, and $10000$.}
\label{fig_rp1}
\end{figure}

\subsubsection{Imaginary} We illustrate this behavior with the example $(a,b) = (1/2,1)$. Figure \ref{fig_rphalf1} plots the zeros of the degree 62 polynomial $p_{1/2,1}(w;2000)$, with restricted imaginary parts for visibility.

\begin{figure}[ht]
\centering
 \includegraphics[scale=.21]{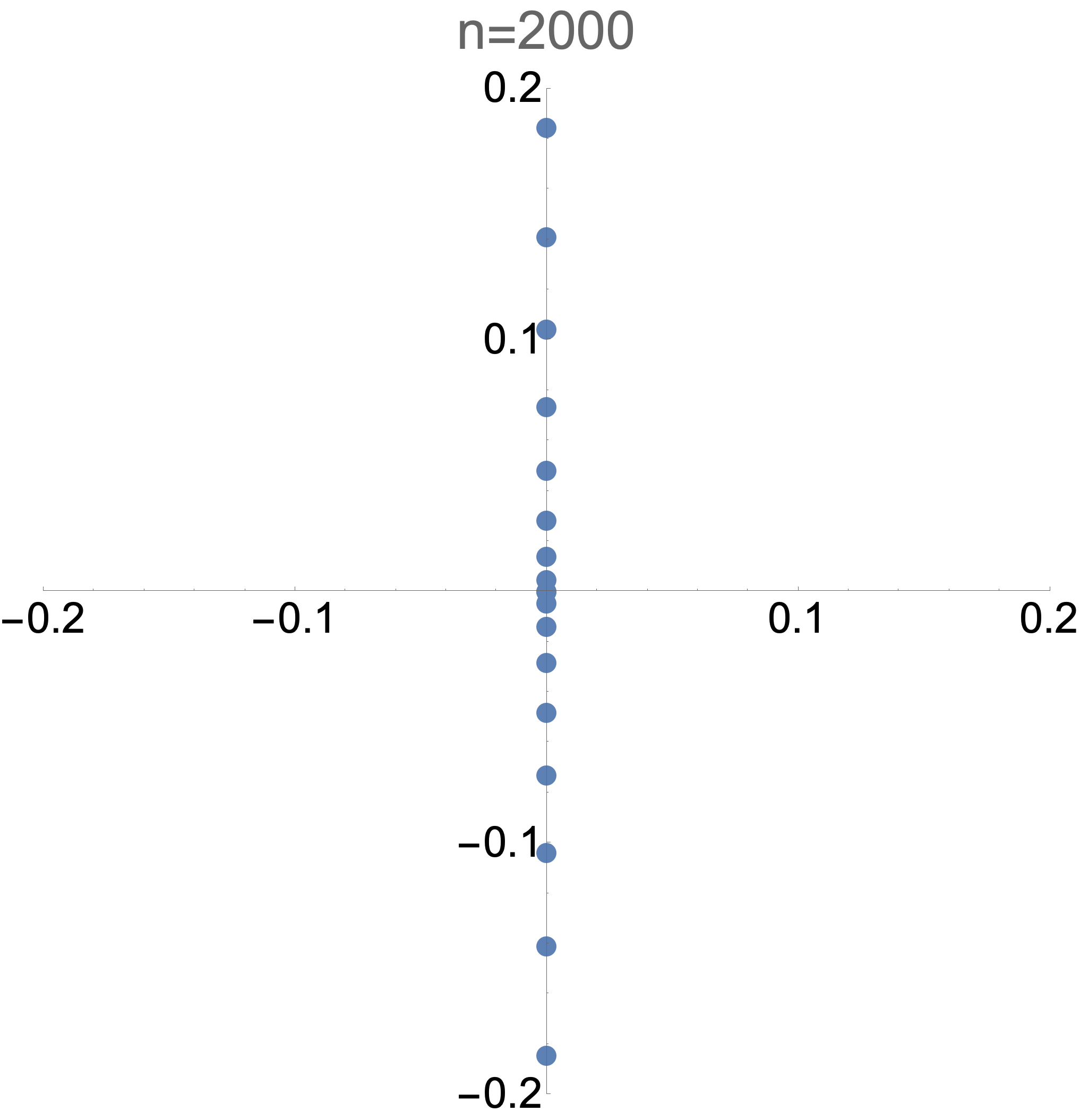}
\caption{Zeros of $p_{1/2,1}(w;2000)$ with restricted imaginary parts.}
\label{fig_rphalf1}
\end{figure}

\subsubsection{Radial segments} We illustrate this behavior with the example pairs $(a,b)$ in the set $\{ (1,1/3), (1,1/4)\}$. Figure~\ref{fig_rprad} plots the zeros of the degree 42 and 44 polynomials $p_{1,1/3}(w;2000)$ and $p_{1,1/4}(w;2000)$, respectively, with restricted moduli for visibility.

\begin{figure}[ht]
\centering
\includegraphics[scale=.22]{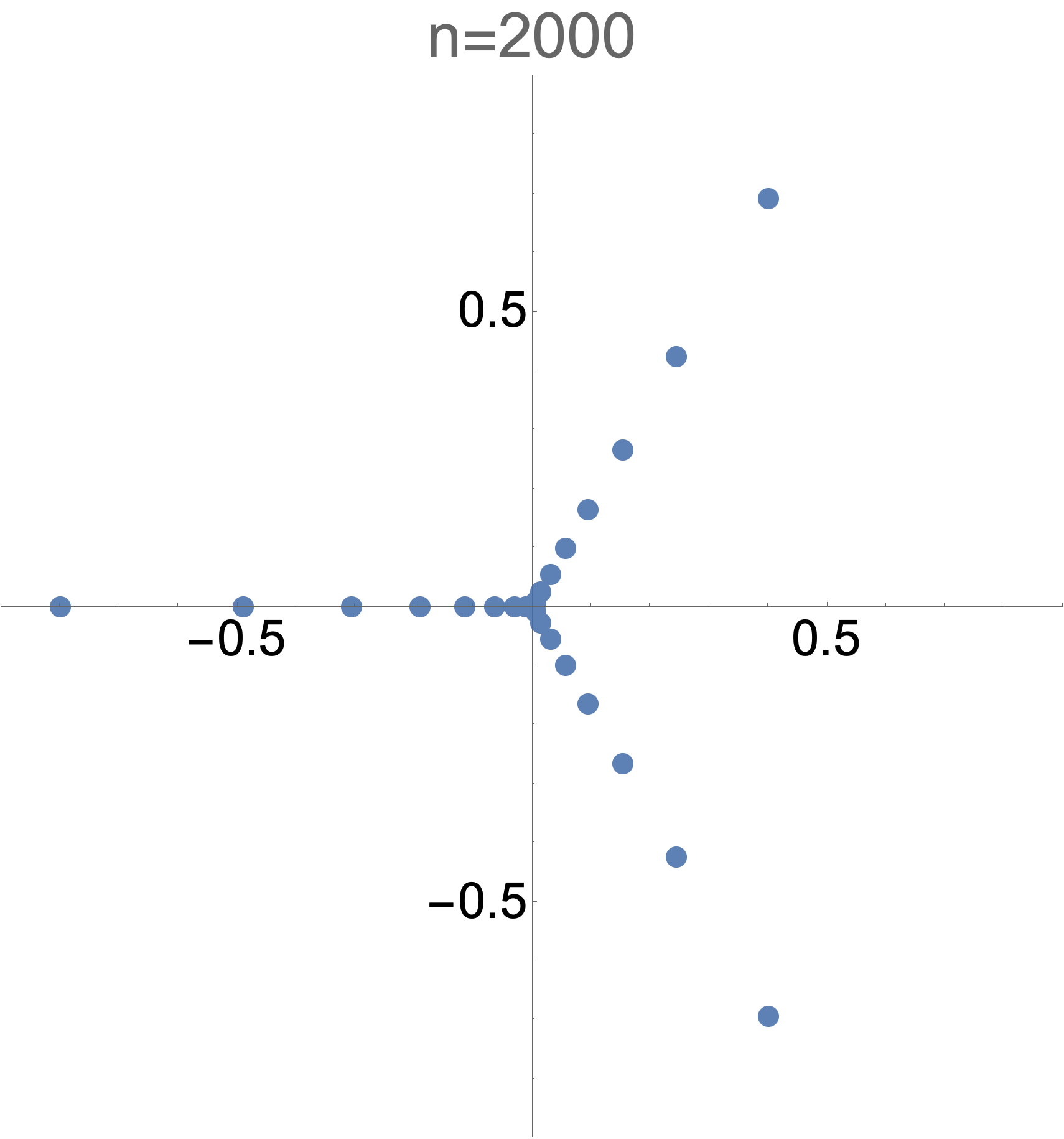} \hspace{.2in}
 \includegraphics[scale=.22]{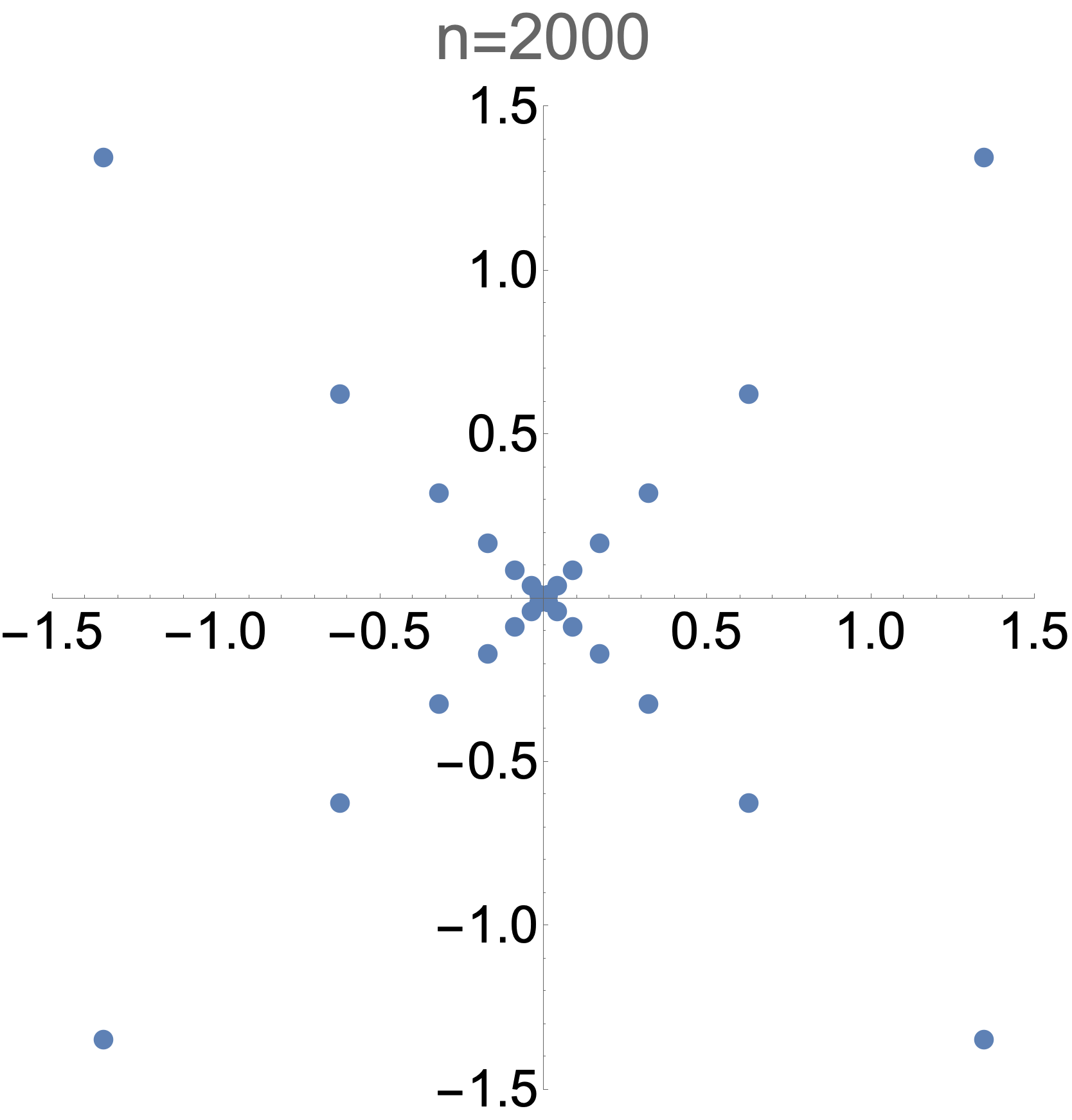}
\caption{Zeros of $p_{1,1/3}(w;2000)$ (left) and $p_{1,1/4}(w;2000)$ (right) with restricted moduli.}
\label{fig_rprad}
\end{figure}
 \subsubsection{Circular}
We illustrate this behavior with the example $(a,b) = (1/3,2/7)$. Figure \ref{fig_rpfrac} plots the zeros of the degree 26 polynomial $p_{1/3,2/7}(w;7114/21)=281936495 w^{26}+567030825181 w^{19}+4450838 w^5$, with restricted moduli for visibility.

\begin{figure}[ht]
\centering
 \includegraphics[scale=.25]{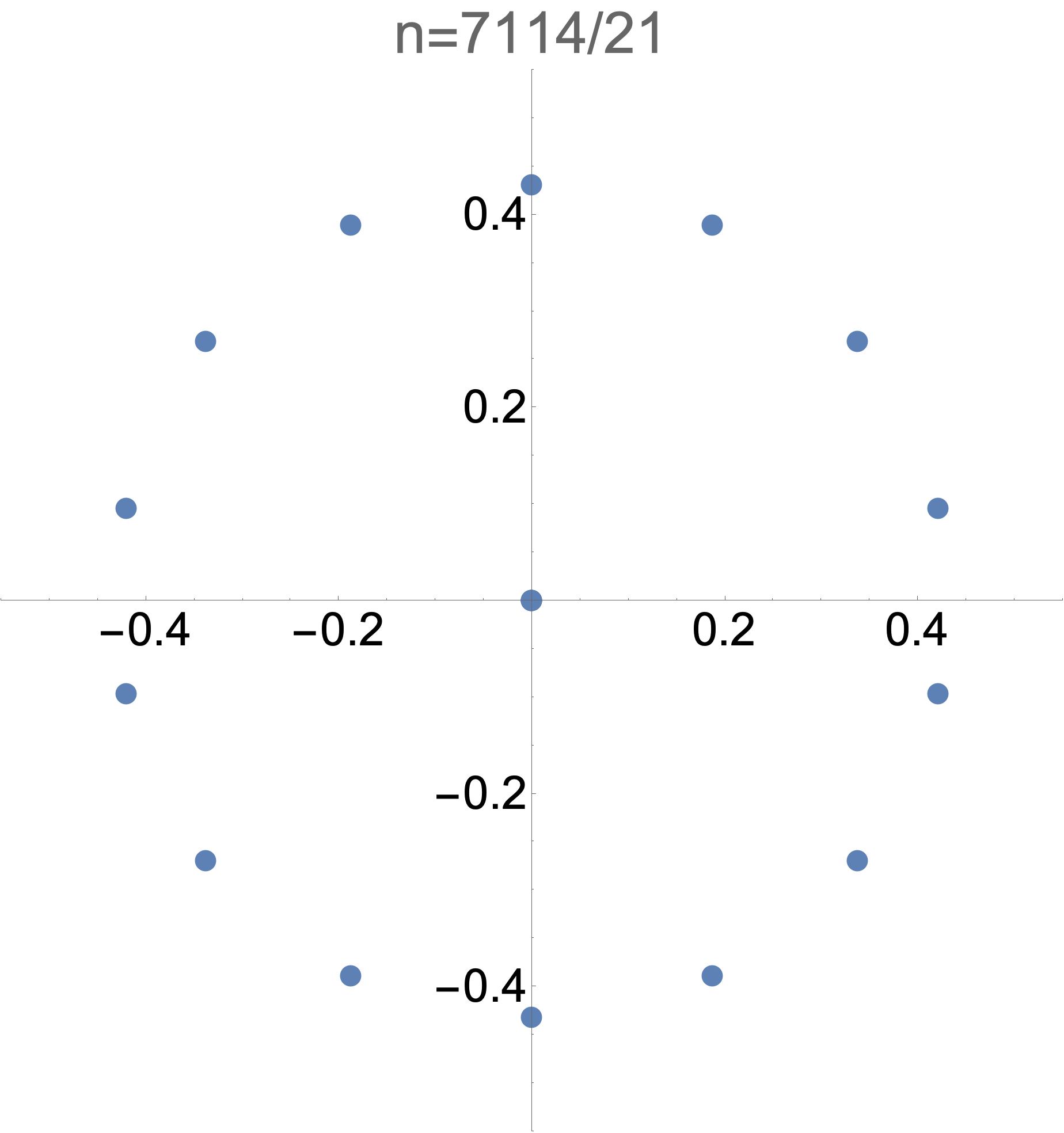}
\caption{Zeros of $p_{1/3,2/7}(w;7114/21)$ with restricted moduli.}
\label{fig_rpfrac}
\end{figure}

\subsection*{Acknowledgments}

We thank the anonymous referees for providing useful feedback which improved the exposition. We also thank Joshua Males for helpful discussions regarding computation of the zeros of hook polynomials. The views expressed in this article are those of the authors and do not reflect the official policy or position of the U.S.\ Naval Academy, Department of the Navy, the Department of Defense, or the U.S.\ Government. The third author is supported by National Science Foundation Grant DMS-2200728.  The fourth author is supported by a grant from the Simons Foundation (853830, LR).

\newpage

 \pdfbookmark[1]{References}{ref}
\LastPageEnding


\begin{thebibliography}{99}
\footnotesize\itemsep=0pt

\bibitem{Anderson}
Anderson J., An asymptotic formula for the {$t$}-core partition function and
 a~conjecture of {S}tanton, \href{https://doi.org/10.1016/j.jnt.2007.10.006}{\textit{J.~Number Theory}} \textbf{128} (2008),
 2591--2615.

\bibitem{Andrews}
Andrews G.E., The theory of partitions, \textit{Cambridge Math. Lib.}, \href{https://doi.org/10.1017/CBO9780511608650}{Cambridge
 University Press}, Cambridge, 1984.

\bibitem{AG}
Andrews G.E., Garvan F.G., Dyson's crank of a~partition, \href{https://doi.org/10.1090/S0273-0979-1988-15637-6}{\textit{Bull. Amer.
 Math. Soc.}} \textbf{18} (1988), 167--171.

\bibitem{BG1}
Boyer R.P., Goh W.M.Y., Partition polynomials: asymptotics and zeros, in Tapas
 in Experimental Mathematics, \textit{Contemp. Math.}, Vol. 457, \href{https://doi.org/10.1090/conm/457/08904}{American
 Mathematical Society}, Providence, RI, 2008, 99--111, \href{https://arxiv.org/abs/0711.1373}{arXiv:0711.1373}.

\bibitem{BG2}
Boyer R.P., Goh W.M.Y., Polynomials associated with partitions: asymptotics and
 zeros, in Special Functions and Orthogonal Polynomials, \textit{Contemp.
 Math.}, Vol. 471, \href{https://doi.org/10.1090/conm/471/09204}{American Mathematical Society}, Providence, RI, 2008,
 33--45, \href{https://arxiv.org/abs/0711.1400}{arXiv:0711.1400}.


\bibitem{BoyerParry1}
Boyer R.P., Parry D.T., On the zeros of plane partition polynomials,
 \href{https://doi.org/10.37236/2026}{\textit{Electron.~J.~Combin.}} \textbf{18} (2011), 30, 26~pages.

\bibitem{BoyerParry3}
Boyer R.P., Parry D.T., Phase calculations for planar partition polynomials,
 \href{https://doi.org/10.1216/RMJ-2014-44-1-1}{\textit{Rocky Mountain~J. Math.}} \textbf{44} (2014), 1--18.

\bibitem{BoyerParry2}
Boyer R.P., Parry D.T., Plane partition polynomial asymptotics,
 \href{https://doi.org/10.1007/s11139-014-9573-8}{\textit{Ramanujan~J.}} \textbf{37} (2015), 573--588, \href{https://arxiv.org/abs/1401.1893}{arXiv:1401.1893}.

\bibitem{BoyerParry4}
Boyer R.P., Parry D.T., Zero attractors of partition polynomials,
 \href{https://arxiv.org/abs/2111.12226}{arXiv:2111.12226}.

\bibitem{BFG}
Bridges W., Franke J., Garnowski T., Asymptotics for the twisted eta-product
 and applications to sign changes in partitions, \href{https://doi.org/10.1007/s40687-022-00355-x}{\textit{Res. Math. Sci.}}
 \textbf{9} (2022), 61, 31~pages, \href{https://arxiv.org/abs/2111.04183}{arXiv:2111.04183}.

\bibitem{BFM}
Bridges W., Franke J., Males J., Sign changes in statistics for plane
 partitions, \href{https://doi.org/10.1016/j.jmaa.2023.127719}{\textit{J.~Math. Anal. Appl.}} \textbf{530} (2024), 127719,
 27~pages, \href{https://arxiv.org/abs/2207.14590}{arXiv:2207.14590}.

\bibitem{BGRT}
Bringmann K., Gomez K., Rolen L., Tripp Z., Infinite families of crank
 functions, {S}tanton-type conjectures, and unimodality, \href{https://doi.org/10.1007/s40687-022-00333-3}{\textit{Res. Math.
 Sci.}} \textbf{9} (2022), 37, 16~pages, \href{https://arxiv.org/abs/2108.12979}{arXiv:2108.12979}.

\bibitem{FMR}
Folsom A., Males J., Rolen L., Equidistribution and partition polynomials,
 \href{https://doi.org/10.1007/s11139-023-00762-w}{\textit{Ramanujan~J.}} \textbf{65} (2024), 1827--1848, \href{https://arxiv.org/abs/2209.15114}{arXiv:2209.15114}.

\bibitem{Garvan}
Garvan F., New combinatorial interpretations of {R}amanujan's partition
 congruences mod~{$5$}, {$7$} and~{$11$}, \href{https://doi.org/10.2307/2001040}{\textit{Trans. Amer. Math. Soc.}}
 \textbf{305} (1988), 47--77.

\bibitem{GKS}
Garvan F., Kim D., Stanton D., Cranks and {$t$}-cores, \href{https://doi.org/10.1007/BF01231493}{\textit{Invent. Math.}}
 \textbf{101} (1990), 1--17.

\bibitem{GranvilleOno}
Granville A., Ono K., Defect zero {$p$}-blocks for finite simple groups,
 \href{https://doi.org/10.1090/S0002-9947-96-01481-X}{\textit{Trans. Amer. Math. Soc.}} \textbf{348} (1996), 331--347.

\bibitem{GORZ}
Griffin M., Ono K., Rolen L., Zagier D., Jensen polynomials for the {R}iemann
 zeta function and other sequences, \href{https://doi.org/10.1073/pnas.1902572116}{\textit{Proc. Natl. Acad. Sci. USA}}
 \textbf{116} (2019), 11103--11110, \href{https://arxiv.org/abs/1902.07321}{arXiv:1902.07321}.

\bibitem{Han}
Han G.N., The {N}ekrasov--{O}kounkov hook length formula: refinement,
 elementary proof, extension and applications, \href{https://doi.org/10.5802/aif.2515}{\textit{Ann. Inst. Fourier
 (Grenoble)}} \textbf{60} (2010), 1--29, \href{https://arxiv.org/abs/0805.1398}{arXiv:0805.1398}.

\bibitem{HR}
Hardy G.H., Ramanujan S., Asymptotic formulaae in combinatory analysis,
 \href{https://doi.org/10.1112/plms/s2-17.1.75}{\textit{Proc. London Math. Soc.}} \textbf{17} (1918), 75--115.

\bibitem{HN19}
Heim B., Neuhauser M., On conjectures regarding the {N}ekrasov--{O}kounkov hook
 length formula, \href{https://doi.org/10.1007/s00013-019-01335-4}{\textit{Arch. Math. (Basel)}} \textbf{113} (2019), 355--366,
 \href{https://arxiv.org/abs/1810.02226}{arXiv:1810.02226}.

\bibitem{HN21}
Heim B., Neuhauser M., On the growth and zeros of polynomials attached to
 arithmetic functions, \href{https://doi.org/10.1007/s12188-021-00241-3}{\textit{Abh. Math. Semin. Univ. Hambg.}} \textbf{91}
 (2021), 305--323, \href{https://arxiv.org/abs/2101.04654}{arXiv:2101.04654}.

\bibitem{HN23}
Heim B., Neuhauser M., Tr\"oger R., Zeros transfer for recursively defined
 polynomials, \href{https://doi.org/10.1007/s40993-023-00480-8}{\textit{Res. Number Theory}} \textbf{9} (2023), 79, 14~pages,
 \href{https://arxiv.org/abs/2304.02694}{arXiv:2304.02694}.

\bibitem{HNW18}
Heim B., Neuhauser M., Weisse A., Records on the vanishing of {F}ourier
 coefficients of powers of the {D}edekind eta function, \href{https://doi.org/10.1007/s40993-018-0125-y}{\textit{Res. Number
 Theory}} \textbf{4} (2018), 32, 12~pages, \href{https://arxiv.org/abs/1808.00431}{arXiv:1808.00431}.

\bibitem{JMOS}
Jin S., Ma W., Ono K., Soundararajan K., Riemann hypothesis for period
 polynomials of modular forms, \href{https://doi.org/10.1073/pnas.1600569113}{\textit{Proc. Natl. Acad. Sci. USA}}
 \textbf{113} (2016), 2603--2608, \href{https://arxiv.org/abs/1601.03114}{arXiv:1601.03114}.

\bibitem{Kostant}
Kostant B., Powers of the {E}uler product and commutative subalgebras of
 a~complex simple {L}ie algebra, \href{https://doi.org/10.1007/s00222-004-0370-7}{\textit{Invent. Math.}} \textbf{158} (2004),
 181--226, \href{https://arxiv.org/abs/math.GR/0309232}{arXiv:math.GR/0309232}.

\bibitem{LP}
Lulov N., Pittel B., On the random {Y}oung diagrams and their cores,
 \href{https://doi.org/10.1006/jcta.1998.2939}{\textit{J.~Combin. Theory Ser.~A}} \textbf{86} (1999), 245--280.

\bibitem{Meinardus}
Meinardus G., Asymptotische {A}ussagen \"uber {P}artitionen, \href{https://doi.org/10.1007/BF01180268}{\textit{Math.~Z.}}
 \textbf{59} (1954), 388--398.

\bibitem{NekOk}
Nekrasov N.A., Okounkov A., Seiberg--{W}itten theory and random partitions, in
 The Unity of Mathematics, \textit{Progr. Math.}, Vol. 244, \href{https://doi.org/10.1007/0-8176-4467-9_15}{Birkh\"auser},
 Boston, MA, 2006, 525--596, \href{https://arxiv.org/abs/hep-th/0306238}{arXiv:hep-th/0306238}.

\bibitem{Parry}
Parry D., A~polynomial variation of {M}einardus' theorem, \href{https://doi.org/10.1142/S1793042115500153}{\textit{Int.~J.
 Number Theory}} \textbf{11} (2015), 251--268, \href{https://arxiv.org/abs/1401.1886}{arXiv:1401.1886}.

\bibitem{Pinsky}
Pinsky M.A., Introduction to {F}ourier analysis and wavelets, \textit{Grad.
 Stud. Math.}, Vol. 102, \href{https://doi.org/10.1090/gsm/102}{American Mathematical Society}, Providence, RI, 2009.

\bibitem{Schlosser}
Schlosser M.J., Bilateral identities of the {R}ogers--{R}amanujan type,
 \href{https://doi.org/10.1090/btran/158}{\textit{Trans. Amer. Math. Soc. Ser.~B}} \textbf{10} (2023), 1119--1140,
 \href{https://arxiv.org/abs/1806.01153}{arXiv:1806.01153}.

\bibitem{Sills}
Sills A.V., An invitation to the {R}ogers--{R}amanujan identities, \href{https://doi.org/10.1201/9781315151922}{Chapman and
 Hall/CRC Press}, New York, 2017.

\bibitem{Tyler}
Tyler M., Asymptotics for $t$-core partitions and {S}tanton's conjecture,
 \href{https://arxiv.org/abs/2406.02982}{arXiv:2406.02982}.

\bibitem{Zhu}
Zhu S., Topological strings, quiver varieties, and {R}ogers--{R}amanujan
 identities, \href{https://doi.org/10.1007/s11139-017-9976-4}{\textit{Ramanujan~J.}} \textbf{48} (2019), 399--421,
 \href{https://arxiv.org/abs/1707.00831}{arXiv:1707.00831}.

\end{thebibliography}
\end{document}